\RequirePackage{fix-cm}

\documentclass[smallextended]{svjour3}       

\smartqed 
\usepackage{graphicx}
\usepackage{amssymb}
\usepackage{units}
\usepackage{amssymb,amsmath}
\usepackage{ntheorem}
\usepackage{graphics,graphicx}
\usepackage{units}
\usepackage{color}
\usepackage{caption}
\usepackage{multirow}
\usepackage{makecell}
\usepackage{enumerate}
\usepackage{soul}
\soulregister\cite7 
\soulregister\citep7 
\soulregister\citet7 
\soulregister\ref7 
\soulregister\cref7 
\soulregister\pageref7 

\begin{document}

\title{Permuted preconditioning for extended saddle point problem arising from Neumann boundary control
\thanks{This work was supported by the National Natural Science Foundation of China (No. 12001022).}
}

\author{Chaojie Wang$^{1\ast}$ \and Xuan Zhang$^1$ \and Xingding Chen$^1$}

\institute{Chaojie Wang($\ast$: Corresponding author)\\
              \email{wangcj2019@btbu.edu.cn}\\
            Xuan Zhang\\
              \email{ikeyan$\_$btbu@163.com}\\
            Xingding Chen \\
              \email{chenxd@th.btbu.edu.cn}\\
\at         1.~School of Mathematics and Statistics, Beijing Technology and Business University, Beijing 100048, China
}

\date{Received: date / Accepted: date}

\maketitle

\begin{abstract}
In this paper, a new block preconditioner is proposed for the saddle point problem arising from the Neumann boundary control problem. In order to deal with the singularity of the stiffness matrix, the saddle point problem is first extended to a new one by a regularization of the pure Neumann problem. Then after row permutations of the extended saddle point problem, a new block triangular preconditioner is constructed based on an approximation of the Schur complement. We analyze the eigenvalue properties of the preconditioned matrix and provide eigenvalue bounds. Numerical results illustrate the efficiency of the proposed preconditioning method.
\end{abstract}

\section{Introduction}
\label{intro}
\hl{The PDE-constrained optimization problem plays an important role in many applications} \cite{mrms}-\cite{ftr}. In recent years, a lot of preconditioning methods have been presented. It has been attracted extensive attentions for preconditioning the saddle point problem arising from the PDE-constrained optimization problem with different PDE constraints, boundary conditions and control types [00]-[00]. In \cite{kmb1}, Mardal, Nielsen, and Nordaas proposed certain block diagonal preconditioners for PDE-constrained optimization with limited observations. These preconditioners were shown to be robust with respect to both the mesh size and the regularization parameter. In \cite{jsw1}, Schoberl and Zulehner constructed a kind of symmetric indefinite preconditioners for control problems with elliptic state equations and distributed control. Specially, the $(1,1)$ block of the corresponding saddle point problems is only positive on the kernel of the $(2,1)$ block. For saddle point problems arising from the PDE-constrained optimization problems with a hyperbolic constraint, Benzi, Haber, and Taralli constructed the block triangular preconditioners with diagonal perturbations of the approximate Hessian \cite{mbe2}. As a comparison of preconditioned Krylov subspace iteration methods, Axelsson, Farouq, and Neytcheva constructed a series of block matrix preconditioners for optimal control problems with Possion and convection-diffusion control, respectively \cite{oas1}. Besides, other preconditioning techniques based on domain decomposition and norm equivalence (see, e.g., Heinkenschloss and Nguyen \cite{mhh1},  Arioli, Kourounis, and Loghin \cite{mad2}), operator methods (see, e.g., Zulehner \cite{wz1}, Elvetun and Nielsen \cite{oeb1}, Gergelits, Mardal, and Nielsen, et al. \cite{tgk1},  Kuchta, Mardal, and Mortensen \cite{mkk1}) have also been investigated to help to solve the PDE-constrained optimization problem more efficiently. In this paper, we focus on building efficient precondioners for the  Neumann boundary control problems.

In \cite{trh1}, Rees, Dollar and Wathen approximated the Schur complement by dropping one of its terms and used this strategy to construct block diagonal preconditioners for the pure Neumann boundary control problem and the distributed control problem with Dirichlet boundary condition, Neumann boundary condition and mixed boundary condition, respectively. The method showed performance independence of the mesh size for the distributed control problem, while the resulting iteration number for the pure Neumann boundary control problem varied as the mesh size \hl{decreased.} Meanwhile, Rees and Stoll also used this strategy to construct block-triangular preconditioners and applied them coupled with the Bramble-Pasciak CG method to saddle point problems arising in PDE-constrained optimization \cite{trm1}. As for the pure Neumann boundary control problem, Pearson and Wathen \cite{jpm3} built a preconditioner on the basis of a matching strategy, which was originally proposed by them for constructing preconditioners for the distributed Poisson control problem \cite{jpa2}. For the distributed Poisson problem, the matching strategy renders the eigenvalues of the preconditioned matrix  bounded by certain intervals independent of the mesh size and the regularization parameter. However, this eigenvalue property does not hold for the pure Neumann boundary control problem and the iteration number was found to be mesh size dependent. 

In the pure Neumann boundary control problem, the singularity of the stiffness matrix is sometimes an obstacle to the construction of efficient preconditioners. In order to tackle this issue, some additional constraints are commonly added on the candidate solution of the state variable. One popular way is to specify the value of the candidate state at a node \cite{pbr1}. In this way, the linear system with the modified stiffness matrix as coefficient matrix can be solved by direct or iterative methods. 
\hl{Besides, there are many papers that deal with the regularization of the saddle point problem and its impact on control or identification problem} \cite{aojw}-\cite{bam2}. On this issue, Bochev and Lehoucq \cite{pbr1} proposed a regularized form of the pure Neumann problem based on a saddle-point Lagrangian formulation and optimization settings. And the solution of the regularized problem is known to be consistent with the original Neumann problem. The idea of this technique will be used in this paper to overcome the singularity of the stiffness matrix.

The contribution of this paper is that a new preconditioning method is proposed for the saddle point problem arising from the pure Neumann boundary control problem. Specifically, in this paper we overcome the singularity of the stiffness matrix by adding a global condition to the candidate solution of the state, which yields a regularized form of the original problem. Based on this strategy, we obtain an extended form of the original saddle point problem in which the extended stiffness matrix is nonsingular. After row permutations of the extended saddle point problem, we build a new block triangular preconditioner based on an approximation of the Schur complement and the preconditioning theory for block matrices. We provide a spectral analysis of the preconditioned matrix. Numerical results show that the proposed preconditioning method outperforms other methods.

The remainder of this paper is organized as follows. In Section 2, we introduce the discretization of the pure Neumann boundary control problem and the corresponding saddle point problem. We also give a brief description of some related preconditioning theories and previous work on this issue. In Section 3, we regularize the pure Neumann problem and derive an extended form of the original saddle point problem. In Section 4, a new preconditioning method for the extended saddle point problem is presented and the eigenvalue properties of the preconditioned matrix is analyzed. In Section 5, some numerical experiments are conducted to illustrate the effectiveness of the proposed preconditioning method. Finally, we draw some conclusions in Section 6.

\section{Problem formulation}\label{pf}
In this paper, we consider the boundary control problem constrained by the Poisson problem with pure Neumann boundary conditions, which is of the following form
\begin{equation}
\label{eq2-1}
\min_{y,u}\frac{1}{2}||y-y_{d}||^{2}_{L_{2}(\Omega)}+\frac{\beta}{2}||u||^{2}_{L_{2}(\partial\Omega)},
\end{equation}
subject to
\begin{equation}
\label{eq2-2}
\left\{
\begin{array}{rclcl}
-\nabla^{2}y&=&f                  &\text{in}&\Omega,\\[1mm]
\displaystyle\frac{\partial y}{\partial n}&=&u & \text{on}&\partial\Omega,
\end{array}
\right.
\end{equation}
where $\Omega\subset\mathbb{R}^{2}$ with boundary $\partial\Omega$, $y$ is the state variable, $u$ is the control variable, $f$ is a known source term, $y_{d}$ is the desired state and $\beta$ is a positive regularization parameter. This PDE-constrained problem is supposed to derive a solution $y$ that approaches $y_{d}$ as close as possible. We note that the \hl{state} equation (\ref{eq2-2}) has a solution if and only if
\begin{equation}
\label{eq2-3}
\int_{\Omega}f d\Omega+\int_{\partial\Omega}u ds=0.
\end{equation}
Under this compatibility condition, the optimal control problem (\ref{eq2-1})-(\ref{eq2-2}) has a unique solution (see Theorem 3.1 in \cite{mhh1} for more details).

The weak formulation of the optimal control problem (\ref{eq2-1})-(\ref{eq2-2}) is: find a solution $(y,u)\in H^{1}(\Omega)\times L_{2}(\partial\Omega)$ to the problem \cite{hed1}
\begin{equation}
\label{eq2-4}
\left\{
\begin{array}{rl}
\min&\displaystyle\frac{1}{2}||y-y_{d}||^{2}_{L_{2}(\Omega)}+\displaystyle\frac{\beta}{2}||u||^{2}_{L_{2}(\partial\Omega)},\\[3mm]
\text{s.t.}&\displaystyle\int_{\Omega}\nabla y\cdot\nabla z d\Omega-\displaystyle\int_{\partial\Omega}u\cdot zds=\displaystyle\int_{\Omega}f\cdot z d\Omega, \quad \forall z\in H^{1}(\Omega),
\end{array}
\right.
\end{equation}
where $L_{2}(\Omega)=\{v:\int_{\Omega}|v|^{2}d\Omega<\infty\}$, $H^{1}(\Omega)=\{v: v\in L_{2}(\Omega), \partial v/\partial x_{i}\in L_{2}(\Omega)\}$.

\subsection{Discretization}\label{section1-1}
Let $V^{h}\subset H^{1}(\Omega)$ be an $n$ dimensional trail space with basis $\{\phi_{1},\cdots,\phi_{n}\}$. Then any $y_{h}\in V^{h}$ can be represented as
$$y_{h}=\sum^{n}_{j=1}Y_{j}\phi_{j},$$
with the coefficient vector $\bold{y}=(Y_{1},\cdots,Y_{n})^{T}$.
Let $W^{h}\subset L_{2}(\partial\Omega)$ be an $m_{B}$ dimensional test space with basis $\{\psi_{1},\cdots,\psi_{m_{B}}\}$ with nonzero support on $\partial\Omega$. Then any $u_{h}\in W^{h}$ can be expressed as
$$u_{h}=\sum^{m_{B}}_{j=1}U_{j}\psi_{j},$$
with the coefficient vector $\bold{u}=(U_{1},\cdots,U_{m_{B}})^{T}$. Normally, both the \emph{optimize-then-discretize} strategy and \emph{discretize-then-optimize} strategy can be utilized to discretize the PDE-constrained optimization problem. As for the optimal control problem constrained by the Poisson equation, these two strategies result in the same linear system since the Laplacian operator is self-adjoint \cite{yq1}. Here the \emph{discretize-then-optimize} strategy is utilized combined with the finite element method and the Lagrange multiplier method to solve the pure Neumann boundary control problem (\ref{eq2-1})-(\ref{eq2-2}).

Using the Galerkin finite element method, the weak formulation (\ref{eq2-4}) is discretized as: find a solution $(y_{h},u_{h})\in V_{h}\times W_{h}$ to the problem
\begin{equation}
\label{eq2-5}
\left\{
\begin{array}{cl}
\min&\displaystyle\frac{1}{2}||y_{h}-y_{d}||^{2}_{L_{2}(\Omega)}+\displaystyle\frac{\beta}{2}||u_{h}||^{2}_{L_{2}(\partial\Omega)},\\[3mm]
\text{s.t.}&\displaystyle\int_{\Omega}\nabla y_{h}\cdot\nabla z_{h} d\Omega-\displaystyle\int_{\partial\Omega}u_{h}\cdot z_{h} ds=\displaystyle\int_{\Omega}f\cdot z_{h} d\Omega, \quad \forall z_{h}\in V^{h}(\Omega).
\end{array}
\right.
\end{equation}
The corresponding matrix form is: find a solution $(\bold{y},\bold{u})\in\mathbb{R}^{n}\times\mathbb{R}^{m_{B}}$ to the problem
\begin{equation}
\label{eq2-6}
\left\{
\begin{array}{cl}
\min&\displaystyle\frac{1}{2}\bold{y}^{T}M\bold{y}-\bold{y}^{T}\bold{b}+\displaystyle\frac{\beta}{2}\bold{u}^{T}M_{b}\bold{u},\\[3mm]
\displaystyle\text{s.t.}& K\bold{y}-N_{b}\bold{u}=\bold{f},
\end{array}
\right.
\end{equation}
where the mass matrix $M$ and the stiffness matrix $K$ have the elements:
$$M_{ij}=\int_{\Omega}\phi_{i}\phi_{j} d\Omega, \quad K_{ij}=\int_{\Omega}\nabla\phi_{i}\nabla\phi_{j} d\Omega, i, j=1,\cdots,n,$$
the boundary mass matrix $M_{b}$ and the matrix $N_{b}$ have the elements:
$$[M_{b}]_{ij}=\int_{\partial\Omega}\psi_{i}\psi_{j} d\Omega, \quad [N_{b}]_{kj}=\int_{\partial\Omega}\phi_{k}\psi_{j} ds, i, j=1,\cdots,m_{B}, k=1,\cdots,n, $$
the vectors $\bold{b}$ and $\bold{f}$ have the elements:
$$b_{j}=\int_{\Omega}y_{d}\phi_{j} d\Omega,\quad f_{j}=\int_{\Omega}f\phi_{j} d\Omega, j=1,\cdots,n.$$

The Lagrange functional of the optimization problem (\ref{eq2-6}) takes the form
$$
L(\bold{y},\bold{u},\bold{p})=\frac{1}{2}\bold{y}^{T}M\bold{y}-\bold{y}^{T}\bold{b}+\frac{\beta}{2}\bold{u}^{T}M_{b}\bold{u}+\bold{p}^{T}(K\bold{y}-N_{b}\bold{u}-\bold{f}),
$$
where $\bold{p}=(P_{1},\cdots,P_{n})$ and $p_{h}=\sum^{n}_{j=1}P_{j}\phi_{j}$ is the finite element approximation of the Lagrange multiplier $p$. Differentiating $L(\bold{y},\bold{u},\bold{p})$ with respect to variables $\bold{y},\bold{u},\bold{p}$, we obtain the first order optimality conditions
\begin{equation}
\label{eq2-7}
\renewcommand\arraystretch{1.2}\left(
\begin{array}{cc|c}
M        & \bold{0}    & K\\
\bold{0} & \beta M_{b} & -N^{T}_{b}\\
\hline
K        & -N_{b}      & \bold{0}
\end{array}
\right)
\left(
\begin{array}{c}
\bold{y}\\
\bold{u}\\
\bold{p}
\end{array}
\right)
=
\left(
\begin{array}{c}
\bold{b}\\
\bold{0}\\
\bold{f}
\end{array}
\right).
\end{equation}
\hl{The weak formulation together with finite element discretization give numerical solutions as approximations of the optimal control problem. This is a standard approximation procedure in PDE-constrained optimization. And we can refer to \cite{mrms}, \cite{ftr} for more details about the convergence property of this discretization scheme. In this paper, we will focus on preconditioning for the resulting linear system.} We note that the stiffness matrix $K$ here is singular. As partitioned in (\ref{eq2-7}), the coefficient matrix is of saddle point structure.

Next we briefly describe some basic preconditioning theories that can be used for the saddle point problem. 

\subsection{Preconditioning theory for block matrices}\label{sec2-2}
Consider the block matrix:
\begin{equation}
\label{eq2-8}
\mathcal{A}=\left(
\begin{array}{cc}
A & B\\
C & \bold{0}
\end{array}
\right)
\end{equation}
with nonsingular submatrix $A$. The Schur complement with respect to $A$ is $S=-CA^{-1}B$. When $C=B^{T}$, the block matrix $\mathcal{A}$ in (\ref{eq2-8}) has the saddle point structure. \hl{In this case, $S=-B^{T}A^{-1}B$ is nonsingular if $B$ has full rank.} As for the block matrix $\mathcal{A}$, the following block diagonal matrix and block upper triangular matrices
$$
\mathcal{P}_{1}=\left(
\begin{array}{cc}
A        & \bold{0}\\
\bold{0} & -S
\end{array}
\right),
\mathcal{P}_{2}=\left(
\begin{array}{cc}
A        & B\\
\bold{0} & \pm S
\end{array}
\right)
$$
are known as its optimal preconditioners. Murphy et al. \cite{mmg1} showed that the preconditioned matrix $\mathcal{P}^{-1}_{1}\mathcal{A}$ has the minimal polynomial $p_{1}(t)=t(t-1)(t^{2}-t-1)$ with degree 4 and the preconditioned matrix $\mathcal{P}^{-1}_{2}\mathcal{A}$ has the minimal polynomial $p_{2}(t)=(t\pm1)(t-1)$ with degree 2. Ipsen \cite{ii1} generalized the results to the more general case:
\begin{equation}
\label{eq2-9}
\mathcal{A}=\left(
\begin{array}{cc}
A & B\\
C & D
\end{array}
\right),
\end{equation}
where $A$ is nonsingular and $D\neq0$. Accordingly, the Schur complement with respect to $A$ is $S=D-CA^{-1}B$. The Schur complement is nonsingular as $\mathcal{A}$ is assumed to be nonsingular. In this case, Ipsen showed that the preconditioned matrix $\mathcal{P}^{-1}_{2}\mathcal{A}$ still has the minimal polynomial $p_{2}(t)=(t\pm1)(t-1)$ with degree 2. Since the preconditioned matrix $\mathcal{P}^{-1}_{2}\mathcal{A}$ is unsymmetric, the generalized minimum residual (GMRES) method is commonly used. According to its convergence property, the GMRES method with the preconditioner $\mathcal{P}_{2}$ should converge in 2 iterations. However, it is often expensive to calculate $S^{-1}$ or to solve the linear system with $S$ as the coefficient matrix. Thus, the optimal preconditioners for $\mathcal{A}$ are normally replaced by their approximations
$$
\widehat{\mathcal{P}}_{1}=\left(
\begin{array}{cc}
\widehat{A} & \bold{0}\\
\bold{0}    & -\widehat{S}
\end{array}
\right),
\widehat{\mathcal{P}}_{2}=\left(
\begin{array}{cc}
\widehat{A} & B\\
\bold{0}    & \pm\widehat{S}
\end{array}
\right),
$$
where $\widehat{A}$ and $\widehat{S}$ are certain suitable approximations of $A$ and $S$. For purpose of reducing the computational cost related to the above preconditioners, $\widehat{S}$ is commonly supposed to have some sparse and structured factorization. These techniques have been applied to constructing efficient preconditioners for many problems. 

\subsection{Previous work}\label{sec2-3}
Recently, a number of preconditioners have been built for the saddle point problem (\ref{eq2-7}) arising from the pure Neumann boundary control problem, which has the block form in (\ref{eq2-8}) with
$$A=\left(
\begin{array}{cc}
M        & \bold{0}\\
\bold{0} & \beta M_{b}
\end{array}
\right),
B=\left(
\begin{array}{c}
K\\
-N^{T}_{b}
\end{array}
\right),
C=\left(
\begin{array}{cc}
K & -N_{b}
\end{array}
\right).
$$
The resulting Schur complement with respect to $A$ is 
$$S=KM^{-1}K+\frac{1}{\beta}N_{b}M^{-1}_{b}N^{T}_{b}.$$
Here we give a brief description of the results proposed by Rees et al. \cite{trh1} and Pearson et al. \cite{jpm3} for this problem.

In \cite{trh1}, Rees et al. presented a block diagonal preconditioner taking the form:
\begin{equation}
\label{eq2-10}
\mathcal{P}^{B}_{D}=\left(
\begin{array}{cc}
\widehat{A}^{B}_{D} & \bold{0}\\
\bold{0}             & \widehat{S}^{B}_{D}
\end{array}
\right),
\end{equation}
where
$$
\widehat{A}^{B}_{D}=\left(
\begin{array}{cc}
\widetilde{M}  & \bold{0} \\
\bold{0}       & \beta \widetilde{M_{b}} \\
\end{array}
\right),
\widehat{S}^{B}_{D}=\widetilde{K}M^{-1}\widetilde{K}^{T}.
$$
In this method, the mass matrices are approximated by a fixed number of Chebyshev iterations. The stiffness matrix is approximated by two algebric multigrid (AMG) V-cycles in $\mathcal{P}^{B}_{D}$. Moreover, the key point of constructing these preconditioners lies in approximating the Schur complement only by its first term, which has a sparse factorization naturally. When these preconditioners are used combined with the minimal residual (MINRES) method, they are found to be mesh size independent while the performance deteriorates as the regularization parameter $\beta$ tends to 0.

On basis of a matching strategy, Pearson et al. \cite{jpm3} proposed a block diagonal preconditioner, which is of the forms:
\begin{equation}
\label{eq2-11}
\widetilde{\mathcal{P}}=\left(
\begin{array}{cc}
\widetilde{A} & \bold{0}\\
\bold{0}          & \widetilde{S}
\end{array}
\right),
\end{equation}
where
$
\widetilde{A}=\left(
\begin{array}{cc}
\widetilde{M}  & \bold{0} \\
\bold{0}       & \beta \widetilde{M_{b}} \\
\end{array}
\right)
$
and the mass matrices are also approximated by some Chebyshev iterations. Besides, the Schur complement is approximated as
$$\widetilde{S}=\bigg(K+\sqrt{\frac{h}{\beta}}M_{\gamma}\bigg)(h\widehat{M_{\gamma}})^{-1}\bigg(K+\sqrt{\frac{h}{\beta}}M_{\gamma}\bigg),$$
where $M_{\gamma}=N_{b}M^{-1}_{b}N^{T}_{b}$,  $\widehat{M_{\gamma}}$ is a matrix given by the matrix $M_{b}$ in the boundary components and a small scalar of order $h$ for all nodes corresponding to the degrees of freedom on the interior, and $h$ denotes the mesh size.  Unlike the matching strategy for the distributed Poisson control problem \cite{jpa2},  $\widetilde{S}^{-1}S$ does not have precise spectrum bounds even though their eigenvalues were illustrated to be bounded heuristically by constants of $\mathcal{O}(1)$. Besides, the required iteration numbers of the MINRES method combined with these two preconditioners were shown to increase as the mesh size $h$ or the regulation parameter $\beta$ turns smaller.

At first, we attempted to precondition the saddle point problem (\ref{eq2-7}) directly by using matrix permutation and Schur complement approximation techniques. But we found the main problem during this procedure is that the stiffness matrix $K$ is singular while the nonsingularity is needed in the Schur complement of the transformed linear system and the corresponding spectral analysis. In order to achieve the above preconditioning techniques, we will extend the saddle point problem in the next section and build new efficient preconditioners on this basis.

\section{Extended system}\label{es}
\subsection{Regularization of the pure Neumann probem}
The Galerkin discretization of the pure Neumann problem:
\begin{equation}
\label{eq3-1}
-\nabla^{2}y=f\quad\text{in}\quad\Omega \quad \text{and}\quad \displaystyle\frac{\partial y}{\partial n}=u\quad\text{on}\quad\partial\Omega
\end{equation}
leads to a linear system as given in (\ref{eq2-6}). Since the stiffness matrix $K$ is singular, some additional constraints should be added on the candidate solution in order to solve the linear system by a direct or iterative \hl{solver. }

\hl{Recall} that under the compatibility condition (\ref{eq2-3}), this pure Neumann problem has solutions. Suppose $y_{0}$ is a solution. As done in \cite{pbr1}, we constrain the candidate solution $y_{0}$ globally by the condition:
\begin{equation}
\label{eq3-2}
\int_{\Omega}y_{0} d\Omega=0.
\end{equation}
Let $y_{0h}\in V^{h}$ be the discretized form of $y_{0}$, which takes the form 
$$y_{0h}=\sum^{n}_{j=1}Y_{0j}\phi_{j}$$
with the coefficient vector $\bold{y_{0}}=(Y_{01},\cdots,Y_{0n})^{T}$. Then the Galerkin discretization of condition (\ref{eq3-2}) yields
$$
\int_{\Omega}y_{0h}d\Omega=\sum^{n}_{j=1}Y_{0j}\int_{\Omega}\phi_{j}d\Omega=\sum^{n}_{j=1}Y_{0j}\omega_{j}=0,
$$
where $\omega_{j}=\int_{\Omega}\phi_{j}d\Omega$ for $j=1,\cdots,n$. Denote $\boldsymbol{\omega}=(\omega_{1},\cdots,\omega_{n})^{T}$, then we have
$$\boldsymbol{\omega}^{T}\bold{y_{0}}=0.$$
Moreover, \hl{for the nodal basis of an FE discretization, it holds} that $\sum^{n}_{j=1}\phi_{j}=1$, the elements of $\boldsymbol{\omega}$ can be computed as
$$\omega_{j}=\int_{\Omega}\phi_{j}\cdot 1d\Omega=\int_{\Omega}\phi_{j}\sum^{n}_{i=1}\phi_{i}d\Omega=\sum^{n}_{i=1}\int_{\Omega}\phi_{i}\phi_{j}d\Omega.$$
Then we have
$$\boldsymbol{\omega}=M\bold{1}.$$
Therefore, a direct Galerkin discretization of the regularized Neumann problem (\ref{eq3-1})-(\ref{eq3-2}) leads to 
\begin{equation}
\label{eq3-3}
K\bold{y_{0}}=\bold{f}+N_{b}\bold{u}\quad \text{and}\quad \boldsymbol{\omega}^{T}\bold{y_{0}}=0,
\end{equation}
which is equivalent to solving the following problem:
$$
\min_{\bold{y_{0}}\in R^{n}}\frac{1}{2}\bold{y^{T}_{0}}K\bold{y_{0}}-\bold{y^{T}_{0}}(\bold{f}+N_{b}\bold{u}) \quad\text{s.t.}\quad \boldsymbol{\omega}^{T}\bold{y_{0}}=0.
$$
Using the Lagrange multiplier method, we can obtain the solution of (\ref{eq3-1})-(\ref{eq3-2}) by solving the following linear system:
\begin{equation}
\label{eq3-4}
\left(
\begin{array}{cc}
K & \boldsymbol{\omega}\\
\boldsymbol{\omega}^{T} & 0
\end{array}
\right)
\left(
\begin{array}{c}
\bold{y_{0}}\\
\lambda
\end{array}
\right)
=
\left(
\begin{array}{c}
\bold{f}+N_{b}\bold{u}\\
\bold{0}
\end{array}
\right),
\end{equation}
where $\lambda$ denotes the Lagrange multiplier. It was shown in \cite{pbr1} that the solution of the regularized Neumann problem (\ref{eq3-1})-(\ref{eq3-2}) coincides with that of the original Neumann problem (\ref{eq3-1}). More importantly, the coefficient matrix in (\ref{eq3-4}) is nonsingular, which can be solved by direct or iterative methods.

\subsection{A new saddle point problem} 
Recall that if $y_{0}$ is a solution of the pure Neumann problem (\ref{eq3-1}), then any solution $y$ takes the form
\begin{equation}\label{eq3-5}y=y_{0}+c, c\in\mathbb{R}.\end{equation}
In order to overcome the singularity of the stiffness matrix $K$, now we consider the regularized optimal control problem
\begin{equation}
\label{eq3-6}
\min_{y_{0},u}\frac{1}{2}||y_{0}+c-y_{d}||^{2}_{L_{2}(\Omega)}+\frac{\beta}{2}||u||^{2}_{L_{2}(\partial\Omega)},
\end{equation}
subject to
\begin{equation}
\label{eq3-7}
\left\{
\begin{array}{rclcl}
-\nabla^{2}y_{0}&=&f                  &\text{in}&\Omega,\\[1mm]
\displaystyle\frac{\partial y_{0}}{\partial n}&=&u & \text{on}&\partial\Omega,\\[1mm]
\displaystyle\int_{\Omega}y_{0} d\Omega&=&0.&&
\end{array}
\right.
\end{equation}
As done in Section 1-1, the matrix form of the weak formulation of the above problem reads: find a solution $(\bold{y_{0}},\bold{u})\in\mathbb{R}^{n}\times\mathbb{R}^{m_{B}}$ to the problem
\begin{equation}
\label{eq3-8}
\left\{
\begin{array}{cl}
\min&\displaystyle\frac{1}{2}(\bold{y_{0}}+c\cdot\bold{1})^{T}M(\bold{y_{0}}+c\cdot\bold{1})-(\bold{y_{0}}+c\cdot\bold{1})^{T}\bold{b}+\displaystyle\frac{\beta}{2}\bold{u}^{T}M_{b}\bold{u},\\[3mm]
\displaystyle\text{s.t.}& \left(
\begin{array}{cc}
K & \boldsymbol{\omega}\\
\boldsymbol{\omega}^{T} & 0
\end{array}
\right)
\left(
\begin{array}{c}
\bold{y_{0}}\\
\lambda
\end{array}
\right)
=
\left(
\begin{array}{c}
\bold{f}+N_{b}\bold{u}\\
\bold{0}
\end{array}
\right).
\end{array}
\right.
\end{equation}
The first order conditions for optimality take the form
\begin{equation}
\label{eq3-9}
\left(
\begin{array}{cccccc}
M                       & \bold{0}            & \bold{0}    & \boldsymbol{\omega}                & K                       & \boldsymbol{\omega} \\
\bold{0}                & 0                   & \bold{0}    & 0                                  & \boldsymbol{\omega}^{T} & 0 \\
\bold{0}                & \bold{0}            & \beta M_{b} & \bold{0}                           & -N^{T}_{b}              & \bold{0} \\
\boldsymbol{\omega}^{T} & 0                   & \bold{0}    &\boldsymbol{\omega^{T}}\bold{1}& \bold{0}                & 0        \\
K                       & \boldsymbol{\omega} & -N_{b}      & \bold{0}                           & \bold{0}                & \bold{0} \\
\boldsymbol{\omega}^{T} & 0                   & \bold{0}    & 0                                  & \bold{0}                & 0 
\end{array}
\right)
\left(
\begin{array}{c}
\bold{y_{0}}\\
\lambda \\
\bold{u}\\
c\\
\bold{p}\\
\pi
\end{array}
\right)
=
\left(
\begin{array}{c}
\bold{b}\\
0       \\
\bold{0}\\
\bold{b^{T}}\bold{1}\\
\bold{f}\\
0
\end{array}
\right),
\end{equation}
where $\pi$ is a Lagrange multiplier. Denote
$$
M_{e}=
\left(
\begin{array}{cc}
M       &\bold{0}\\
\bold{0}&0
\end{array}
\right),
K_{e}=
\left(
\begin{array}{cc}
K                       &\boldsymbol{\omega}\\
\boldsymbol{\omega}^{T} &0
\end{array}
\right),
Z_{e}=
\left(
\begin{array}{cc}
\bold{0}&\boldsymbol{\omega}\\
\bold{0}&0
\end{array}
\right),
$$
$$
M_{be}=
\left(
\begin{array}{cc}
\beta M_{b}  &\bold{0}\\
\bold{0}     &\boldsymbol{\omega^{T}}\bold{1}
\end{array}
\right),
N_{be}=
\left(
\begin{array}{cc}
N_{b}   &\bold{0}\\
\bold{0}&0
\end{array}
\right),
$$
and
$$
\bold{y}_{e}=
\left(
\begin{array}{c}
\bold{y_{0}}\\
\lambda
\end{array}
\right),
\bold{u}_{e}=
\left(
\begin{array}{c}
\bold{u}\\
c
\end{array}
\right),
\bold{p}_{e}=
\left(
\begin{array}{c}
\bold{p}\\
\pi
\end{array}
\right),
$$
$$
\bold{b}_{e}=
\left(
\begin{array}{c}
\bold{b}\\
0
\end{array}
\right),
\bold{z_{e}}=
\left(
\begin{array}{c}
\bold{0}\\
\bold{b^{T}}\cdot\bold{1}
\end{array}
\right),
\bold{f}_{e}=
\left(
\begin{array}{c}
\bold{f}\\
0
\end{array}
\right),
$$
then we have
\begin{equation}
\label{eq3-10}
\renewcommand\arraystretch{1.2}\left(
\begin{array}{cc|c}
M_{e}        & Z_{e}       & K_{e}\\
Z^{T}_{e}    & M_{be} & -N^{T}_{be}\\
\hline
K_{e}        & -N_{be}     & \bold{0}
\end{array}
\right)
\left(
\begin{array}{c}
\bold{y}_{e}\\
\bold{u}_{e}\\
\bold{p}_{e}
\end{array}
\right)
=
\left(
\begin{array}{c}
\bold{b}_{e}\\
\bold{z_{e}}\\
\bold{f}_{e}
\end{array}
\right),
\end{equation}
where $K_{e}$ is referred to as the extended stiffness matrix in this paper. 
As partitioned in (\ref{eq3-10}), the extended coefficient matrix still has the saddle point structure. In this paper, we obtain the numerical solution of the pure Neumann boundary control problem by solving this saddle point problem. In the next section, we will focus on building an efficient preconditioner for this saddle point problem.

\section{Preconditioning}\label{np}
In this section, we propose a new preconditioning method for the saddle point problem (\ref{eq3-10}). We also analyze the eigenvalue properties of the preconditioned matrix.

\subsection{A new preconditioning method}
In a different way from the work introduced in Section 2-3 that is based on the preconditioning theory for the matrix in (\ref{eq2-8}), we propose a new preconditioning method for the saddle point problem (\ref{eq3-10}) on basis of row permutations and the preconditioning theory for the more general case (\ref{eq2-9}). 

At first, we transform row blocks of (\ref{eq3-10}) to obtain its equivalent form
\begin{equation}
\label{eq4-1}
\underbrace{
\renewcommand\arraystretch{1.2}
\left(
\begin{array}{c|cc}
K_{e}        & -N_{be}     & \bold{0}\\
\hline
\displaystyle Z^{T}_{e}    & M_{be}      & \displaystyle -N^{T}_{be}\\
M_{e}        & Z_{e}       & K_{e}
\end{array}
\right)
}_{\mathcal{A}}
\left(
\begin{array}{c}
\bold{y}_{e}\\
\bold{u}_{e}\\
\bold{p}_{e}
\end{array}
\right)
=
\left(
\begin{array}{c}
\bold{f}_{e}\\
\bold{z_{e}}\\
\bold{b}_{e}
\end{array}
\right).
\end{equation}
The coefficient matrix has the structure of $\mathcal{A}$ in (\ref{eq2-9}) with
$$
A=K_{e},
B=\left(
\begin{array}{cc}
-N_{be} & \bold{0}
\end{array}
\right),
C=\left(
\begin{array}{c}
Z^{T}_{e}\\
M_{e}
\end{array}
\right),
D=\left(
\begin{array}{cc}
M_{be}      & -N^{T}_{be}\\
Z_{e}       & K_{e}
\end{array}
\right).
$$
The resulting Schur complement of the linear system takes the form
\begin{equation}
\label{eq4-2}
S=\left(
\begin{array}{cc}
M_{be}+Z^{T}_{e}K^{-1}_{e}N_{be} & -N^{T}_{be}\\
Z_{e}+M_{e}K^{-1}_{e}N_{be}      & K_{e}
\end{array}
\right).
\end{equation}
Denote
$$
K^{-1}_{e}=
\left(
\begin{array}{cc}
J            & \bold{v}\\
\bold{v}^{T} & a
\end{array}
\right),
$$
where $J\in\mathbb{R}^{n\times n}, \bold{v}\in\mathbb{R}^{n}$ and $a\in\mathbb{R}$. Based on the fact $K_{e}K^{-1}_{e}=I$, it is easy to derive that
$$\boldsymbol{\omega^{T}}J=\bold{0}, \quad \boldsymbol{\omega^{T}}\bold{v}=1.$$
Immediately, we have
\begin{equation}
\label{eq4-3}
Z^{T}_{e}K^{-1}_{e}N_{be}=\bold{0}.
\end{equation}
We approximate the Schur complement as
\begin{equation}
\label{eq4-4}
\widehat{S}=\left(
\begin{array}{cc}
M_{be}       & -N^{T}_{be}\\
\bold{0}     & K_{e}
\end{array}
\right).
\end{equation}
Then based on the preconditioning theory for the more general case (\ref{eq2-9}), we can take the following block triangular matrix
\begin{equation}
\label{eq4-5}
\widehat{\mathcal{P}}_{2}=\left(
\begin{array}{cc}
A & B\\
0 & \widehat{S}
\end{array}
\right)
=
\renewcommand\arraystretch{1.2}\left(
\begin{array}{c|cc}
K_{e}        & -N_{be}     & \bold{0}\\
\hline
\bold{0}     & M_{be}      & -N^{T}_{be}\\
\bold{0}     & \bold{0}    & K_{e}
\end{array}
\right)
\end{equation}
as \hl{a preconditioner} for $\mathcal{A}$. The linear system $\widehat{\mathcal{P}}_{2}\bold{g}=\bold{d}$ with vectors $\bold{g}=(\bold{g}^{T}_{1}, \bold{g}^{T}_{2}, \bold{g}^{T}_{3})^{T}$ and $\bold{d}=(\bold{d}^{T}_{1}, \bold{d}^{T}_{2}, \bold{d}^{T}_{3})^{T}$ can be solved step by step as follows:
\begin{equation}
\label{eq4-6}
\left\{
\begin{array}{rll}
K_{e}\bold{g}_{3}&=&\bold{d}_{3},\\
M_{be}\bold{g}_{2}&=&\bold{d}_{2}+N^{T}_{be}\bold{g}_{3},\\
K_{e}\bold{g}_{1}&=&\bold{d}_{1}+N_{be}\bold{g}_{2}.\\
\end{array}
\right.
\end{equation}
Since the extended stiffness matrix $K_{e}$ and boundary mass matrix $M_{be}$ are nonsingular and sparse, the solutions of the linear systems in (\ref{eq4-6}) can be computed at low cost using direct methods or iterative methods. In this paper, the GMRES method will be used combined with the preconditioner $\widehat{\mathcal{P}}_{2}$ for solving (\ref{eq4-1}) as the preconditioned matrix $\widehat{\mathcal{P}}^{-1}_{2}\mathcal{A}$ is unsymmetric. Besides, we remark that the above preconditioning method can also be applied to the mixed boundary control problem. In that case, it is no need to do the extension as in Section 3 since the stiffness matrix is already nonsingular originally.

\subsection{Spectral analysis of the preconditioner}\label{sec4-2}
In this subsection, we analyze the spectrum of the preconditioned matrix $\widehat{\mathcal{P}}^{-1}_{2}\mathcal{A}$, which is the same as that of $\mathcal{A}\widehat{\mathcal{P}}^{-1}_{2}$. Based on (\ref{eq4-3}), it is noted that
$$
\mathcal{A}\widehat{\mathcal{P}}^{-1}_{2}=\left(
\begin{array}{cc}
I       & \bold{0}\\
CA^{-1} &S\widehat{S}^{-1}
\end{array}
\right),
$$
thus 1 is an eigenvalue of $\mathcal{A}\widehat{\mathcal{P}}^{-1}_{2}$ with multiplicity $n+1$ and with eigenvectors $e_{j}\in \mathbb{R}^{2n+m_{B}+2}, j=1,\cdots, n+1$. Besides, consider the eigenvalue problem
\begin{equation}
\label{eq4-7}
S\bold{x}=\mu \widehat{S}\bold{x}, 
\end{equation}
where $\bold{x}\neq 0\in \mathbb{R}^{n+m_{B}+2}$ is the eigenvector with \hl{$\bold{x}^{T}=(\bold{x}^{T}_{1},\bold{x}^{T}_{2})$.} Substituting (\ref{eq4-2})-(\ref{eq4-4}) into (\ref{eq4-7}), we have
$$
\renewcommand\arraystretch{1.2}
\left\{
\begin{array}{lll}
(1-\mu)(M_{be}\bold{x}_{1}-N^{T}_{be}\bold{x}_{2})   & = &0,\\
(Z_{e}+M_{e}K^{-1}_{e}N_{be})\bold{x}_{1}+(1-\mu)K_{e}\bold{x}_{2}& = &0.
\end{array}
\right.
$$

\hl{It is noted that $\mu=1$, $\bold{x}_{1}=\bold{0}, \bold{x}_{2}\in \mathbb{R}^{n+1}$ is a solution.} Therefore, 1 is an eigenvalue of $S\widehat{S}^{-1}$ with multiplicity $n+1$ and with eigenvectors $e_{j}\in \mathbb{R}^{n+m_{B}+2}, j=m_{B}+2,\cdots, n+m_{B}+2$. As a result of the above analysis, we obtain that 1 is an eigenvalue of the preconditioned matrix $\widehat{\mathcal{P}}^{-1}_{2}\mathcal{A}$ with multiplicity $2n+2$.

If $\mu\neq1$, then
$$
(1-\mu)M_{be}\bold{x}_{1}+N^{T}_{be}K^{-1}_{e}(Z_{e}+M_{e}K^{-1}_{e}N_{be})\bold{x}_{1}=0.
$$
According to (\ref{eq4-3}), $N^{T}_{be}K^{-1}_{e}Z_{e}=\bold{0}$. Since the extended boundary mass matrix $M_{be}$ is symmetric positive definite and the matrix $N^{T}_{be}K^{-1}_{e}M_{e}K^{-1}_{e}N_{be}$ is symmetric, the eigenvalue $\mu$ is real. On the other hand, we have
$$
(Z_{e}+M_{e}K^{-1}_{e}N_{be})M^{-1}_{be}N^{T}_{be}\bold{x}_{2}+(1-\mu)K_{e}\bold{x}_{2}=0.
$$
Note that $Z_{e}M^{-1}_{be}N^{T}_{be}=\bold{0}$. Then
\begin{equation}
\label{eq4-8} 
\mu=1+\frac{\bold{x}^{T}_{2}K^{-1}_{e}M_{e}K^{-1}_{e}N_{be}M^{-1}_{be}N^{T}_{be}\bold{x}_{2}}{\bold{x}^{T}_{2}\bold{x}_{2}}.
\end{equation}

Denote $T=K^{-1}_{e}M_{e}K^{-1}_{e}N_{be}M^{-1}_{be}N^{T}_{be}$, next let us analyze its eigenvalues. Notice that if the nodes are ordered such that all the boundary nodes are listed first followed by the interior nodes, then $M_{r}=N_{be}M^{-1}_{be}N^{T}_{be}$ has the block structure:
\begin{equation}
\label{eq4-9}
M_{r}=
\left(
\begin{array}{cc}
M_{b}   &\bold{0}\\
\bold{0}&\bold{0}
\end{array}
\right).
\end{equation}
Moreover, the mass matrix $M$ and the boundary mass matrix $M_{b}$ are normally considered in lumped forms, which are diagonal matrices that are spectrally equivalent to scalar matrices including the mesh size $h$ \cite{hed1}. If ignoring all the multiplicative constants,  \hl{$M_{e}$ and $M_{r}$ are spectrally equivalent to matrices}
\begin{equation}
\label{eq4-10}
\left(
\begin{array}{cc}
h^{2}I_{n} &\bold{0}\\
\bold{0}   &\bold{0}
\end{array}
\right),
\left(
\begin{array}{cc}
hI_{m_{B}}   &\bold{0}\\
\bold{0}         &\bold{0}
\end{array}
\right),
\end{equation}
\hl{respectively. Then $K^{-1}_{e}M_{e}K^{-1}_{e}$ is spectrally equivalent to the following matrix}
$$
\left(
\begin{array}{cc}
J            & \bold{v}\\
\bold{v}^{T} & a
\end{array}
\right)
\left(
\begin{array}{cc}
h^{2}I_{n} &\bold{0}\\
\bold{0}   &\bold{0}
\end{array}
\right)
\left(
\begin{array}{cc}
J            & \bold{v}\\
\bold{v}^{T} & a
\end{array}
\right)
=
\left(
\begin{array}{cc}
h^{2}J^{2}         & h^{2}J\bold{v}\\
h^{2}\bold{v}^{T}J & h^{2}\bold{v}^{T}\bold{v}
\end{array}
\right).
$$
\hl{Thus $T=K^{-1}_{e}M_{e}K^{-1}_{e}M_{r}$ is spectrally equivalent to}
\begin{equation}
\label{eq4-11}
\left(
\begin{array}{cc}
h^{3}J^{2}_{m_{B}} & \bold{0}\\
\ast               & \bold{0}
\end{array}
\right),
\end{equation}
where $J_{m_{B}}\in\mathbb{R}^{m_{B}\times m_{B}}$ is an leading principal minor of $J$. It follows that the matrix $T$ is spectrally equivalent to the matrix $h^{3}J^{2}_{m_{B}}$. Moreover, note that the matrix $J$ is symmetric and $J_{m_{B}}$ is a principal \hl{submatrix.} According to the Cauchy's interlace theorem (see \cite{xj1}, Theorem 1.2), immediately we have the eigenvalue relations between $J_{m_{B}}$ and $J$ as the following
\begin{equation}
\label{eq4-12}
\lambda_{\min}(J)\leq\lambda_{\min}(J_{m_{B}})\leq\lambda_{\max}(J_{m_{B}})\leq\lambda_{\max}(J).
\end{equation}
Now we consider the relation between the eigenvalues of $J$ and those of $K$, which is presented in Lemma 4-1.
\begin{lemma}\label{lemma4-1}
The nonzero eigenvalues of $K$ and $J$ are mutually reciprocal with the same eigenvectors.
\end{lemma}
\begin{proof}
Consider the eigenvalue problem
$$K_{e}\bold{x}=\lambda\bold{x},$$
where $\lambda\neq0$ is the eigenvalue and $\bold{x}$ is the corresponding eigenvector with $\bold{x}=(\bold{x}^{T}_{1},x_{2})^{T}, \bold{x}_{1}\in\mathbb{R}^{n}, x_{2}\in\mathbb{R}$. It derives
$$K^{-1}_{e}\bold{x}=\frac{1}{\lambda}\bold{x}.$$
Then 
$$
\left\{
\begin{array}{ccc}
K\bold{x}_{1}+x_{2}\boldsymbol{\omega}& = & \lambda\bold{x}_{1},\\
\boldsymbol{\omega}^{T}\bold{x}_{1}   & = & \lambda x_{2},\\
J\bold{x}_{1}+x_{2}\bold{v}           & = & \frac{1}{\lambda}\bold{x}_{1},\\
\bold{v}^{T}\bold{x}_{1}+ax_{2}       & = & \frac{1}{\lambda} x_{2}.
\end{array}
\right.
$$
If $x_{2}=0$, we have
$$
\left\{
\begin{array}{ccc}
K\bold{x}_{1}                       & = & \lambda\bold{x}_{1},\\
\boldsymbol{\omega}^{T}\bold{x}_{1} & = & 0,\\
J\bold{x}_{1}                       & = & \frac{1}{\lambda}\bold{x}_{1},\\
\bold{v}^{T}\bold{x}_{1}            & = & 0.
\end{array}
\right.
$$
Since $\text{dim}(\text{ker}(\boldsymbol{\omega}^{T}))=n-1$ and $K$ has an eigenvalue 0, the nonzero eigenvalues of $K$ and its inverse are all included in this case. Moreover, it can be easily seen from the above equations that the nonzero eigenvalues of $K$ and $J$ are mutually reciprocal and the corresponding eigenvectors are the same.
\end{proof}

Recall that the stiffness matrix $K$ approximates the scaled identity matrix in the sense that \cite{hed1}
$$d_{1}h^{2}\leq\frac{\bold{v}^{T}K\bold{v}}{\bold{v}^{T}\bold{v}}\leq d_{2}, \forall \bold{v}\neq 0\in \mathbb{R}^{n},$$
where $d_{1}$ and $d_{2}$ are constants independent of the mesh size $h$. Based on Lemma 4-1, we have
$$
\frac{1}{d_{2}} \leq \frac{\bold{v}^{T}J\bold{v}}{\bold{v}^{T}\bold{v}}\leq \frac{h^{-2}}{d_{1}}, \forall \bold{v}\neq 0\in \mathbb{R}^{n}.
$$
Then, according to (\ref{eq4-12}) ,
\begin{equation}
\label{eq4-13}
\frac{h^{3}}{d^{2}_{2}}\leq \frac{\bold{z}^{T}(h^{3}J^{2}_{m_{B}})\bold{z}}{\bold{z}^{T}\bold{z}}\leq \frac{h^{-1}}{d^{2}_{1}}, \forall \bold{z}\neq 0\in \mathbb{R}^{m_{B}}.
\end{equation}
Combining (\ref{eq4-8}), (\ref{eq4-11}) and (\ref{eq4-13}), we have the eigenvalue bound of the preconditioned matrix $\widehat{\mathcal{P}}^{-1}_{2}\mathcal{A}$ as
\begin{equation}
\label{eq4-14}
1+\frac{1}{\beta}ch^{3}\leq\mu_{0}\leq 1+ \frac{1}{\beta}dh^{-1},
\end{equation}
where $c$ and $d$ are constants independent of the regularization parameter $\beta$ and mesh size $h$. 

The above results are summarized in Proposition 4-2.
\begin{proposition}\label{prop4-2}
If the block triangular matrix $\widehat{\mathcal{P}}_{2}$ in (\ref{eq4-5}) is taken as a preconditioner for the matrix $\mathcal{A}$ in (\ref{eq4-1}), then 1 is an eigenvalue of the preconditioned matrix $\widehat{\mathcal{P}}^{-1}_{2}\mathcal{A}$ with multiplicity $2n+2$ and the other $m_{B}$ eigenvalues are bounded by $[1+\frac{1}{\beta}ch^{3}, 1+\frac{1}{\beta}dh^{-1}]$, where $c$ and $d$ are constants independent of the regularization parameter $\beta$ and mesh size $h$.
\end{proposition}

\hl{It is noticed that similar results are presented in} \cite{abat}-\cite{gbog}. \hl{Battermann has designed block preconditioners directly by taking fully advantage of matrix structure \cite{abat}. In this paper, while, the block preconditioner is built based on Ipen's preconditioning framework \cite{ii1}.  The spectral analysis here involves approximation of Schur complement and properties of $M$, $K$ and $M_{b}$. The result here is consistent with that in \cite{abat}. }

\hl{According to Proposition 4-2, the eigenvalue interval grows as $\beta\rightarrow 0$ and $h\rightarrow 0$.} Notice that the preconditioned system is non-symmetric and the GMRES method is used in this paper. Unlike the MINRES method or the Bramble-Pasciak-CG method, the spectral analysis is not sufficient to determine the convergence of the GMRES method (see \cite{mvz} for more details). We remark that the eigenvalue bound provided here, while, can help us to gain a better insight into the property of the proposed preconditioner. Actually, the numerical results in the next section show the good performance of the proposed preconditioner.

\section{Numerical results}\label{nm}
In this section, we illustrated the efficiency of the proposed preconditioning method by comparing its performance with other existing methods on solving the saddle point problem arising from the pure Neumann boundary control problem. All the experiments were conducted using MATLAB R2018a on a Linux computer with Intel Core i5-7300HQ at 2.50 GHz CPUs and 8 GB of RAM.

In every example, the domain was $\Omega=(0,1)^{2}$ and the source term was $f=0$. The regularization parameter took the values $\beta=10^{-2}, 10^{-4}, 10^{-6}, 10^{-8}$. For purpose of discretizing these problems, the $\bold{P1}$ basis functions on a quasi-uniform triangulation of $\Omega$ were used for the state, control and adjoint variables. 

In order to obtain numerical solutions of the pure Neumann boundary control problem, we solved the extended linear system (\ref{eq4-1}) by the GMRES method with the proposed preconditioner $\widehat{\mathcal{P}}_{2}$ in (\ref{eq4-5}), which was denoted as GMRES($\widehat{\mathcal{P}}_{2}$). Besides, we took
\begin{equation}
\label{eq5-1}
\widehat{\mathcal{P}}_{I}
=\left(
\begin{array}{c|cc}
K_{e}        & -N_{be}     & \bold{0}\\
\hline
\bold{0}     & M_{be}     & -N^{T}_{be}\\
\bold{0}     & \bold{0}    & I
\end{array}
\right)
\end{equation}
as a reference and denote the corresponding method as GMRES($\widehat{\mathcal{P}}_{I}$). This is done in a sense that the extened stiffness matrix $K_{e}$ in the Schur complement approximation (\ref{eq4-4}) was merely approximated by the identity matrix. As comparisons, we also used the MINRES method combined with the preconditioner in (\ref{eq2-10}) and (\ref{eq2-11}) to solve the saddle point problem (\ref{eq2-7}) as other ways to obtain numerical solutions of the pure Neumann boundary control problem \cite{trh1}, \cite{jpm3}. Accordingly, these two methods were denoted as MINRES($\mathcal{P}^{B}_{D}$) and MINRES($\widetilde{\mathcal{P}}$).

The linear systems with these preconditioners as coefficient matrices were solved step by step in a similar way as in (\ref{eq4-6}).  In all preconditioners, a solve with the mass matrix $M$ was approximated by 20 steps of relaxed Jacobi accelerated by the Chebyshev semi-iteration. And a solve with the stiffness matrix $K$ was approximated by three AMG V-cycles, which is implemented by the HSL package HSL$\_$MI20 via a MATLAB interface) \cite{jmjj}. \hl{In actual AMG procedure, the stiffness matrix $K$ was replaced with $\widehat{K}$ for better performance. $\widehat{K}$ is empirically the same as $K$ except for its last row and column with elements $(0,\cdots,0,1)$ and the corresponding transpose.  For preconditioners $\widehat{\mathcal{P}}_{2}$ and $\widehat{\mathcal{P}}_{I}$, a solve with the extended stiffness matrix $K_{e}$ is needed. But usually it is not easy for the AMG method to be applied directly. We address this issue by the application of the AMG method on $\widehat{K}$ instead.} In detail, for  $K_{e}\bold{x}=\bold{v}$  with $\bold{x}=(\bold{x}^{T}_{1},x_{2})^{T}$ and $\bold{v}=(\bold{v}^{T}_{1},v_{2})^{T}$, we approximated the solve as follows
\begin{equation}
\label{eq5-2}
\left\{
\begin{array}{rll}
x_{2}            &  =  & (\bold{v}^{T}_{1}\widehat{K}^{-1}\boldsymbol{\omega}-v_{2})/\boldsymbol{\omega}^{T}\widehat{K}^{-1}\boldsymbol{\omega},\\
\bold{x}_{1}&  =  & \widehat{K}^{-1}\bold{v}_{1}-x_{2}\widehat{K}^{-1}\boldsymbol{\omega}.\\
\end{array}
\right.
\end{equation}
Unless otherwise indicated, both the GMRES method and MINRES method were terminated when the relative residual in the 2-norm reached a desired tolerance $10^{-6}$. In the following result tables, the iteration numbers were exhibited followed by the CPU time (in seconds) in the brackets. A dash line implies that the corresponding method needs more than 500 iteration steps.

\subsection{Example 1}\label{sec5-1}
Consider the pure Neumann boundary control problem (\ref{eq2-1})-(\ref{eq2-2}) with the desired state
$$
y_{d}=\left\{
\begin{array}{c}
\begin{array}{ccl}
1,&x\leq\frac{1}{2}, y\leq\frac{1}{2},\\
0,&\text{elsewhere},
\end{array}
\end{array}
\right.
$$
which is illustrated in Fig. 1.

\begin{figure}[!b]
  \centering
  \includegraphics[height=4.5cm]{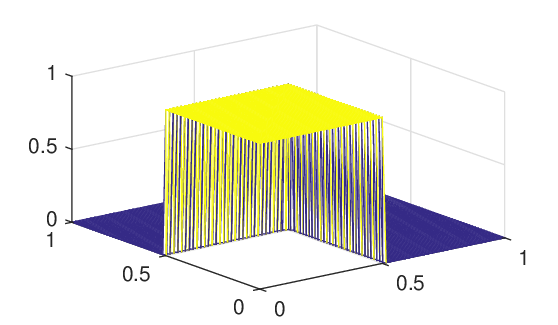}
  \caption{Desired state of Example 1.}
  \label{fig:state1}
\end{figure}

\begin{figure}[!b]
\centering
{\includegraphics[width=\textwidth]{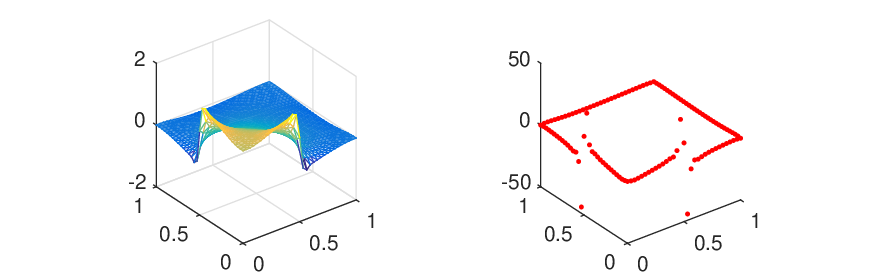}\label{fig:2}}
\caption{Computed state (left) and control (right) of the GMRES($\widehat{\mathcal{P}}_{2}$) method for Example 1 with DoF=2306 and $\beta=10^{-6}$.}
\label{fig:solution1}
\end{figure}

\begin{table}[!b]
\renewcommand\arraystretch{1.5}
\scriptsize
\centering
\caption{\protect\\Comparison of iteration numbers and CPU time for solving Example 1 with $\beta=10^{-2}$}
\begin{tabular}{c|c|c|c|c}
\hline
DoF\verb|\| Method & MINRES($\mathcal{P}^{B}_{D}$) \cite{trh1} & MINRES($\widetilde{\mathcal{P}}$) \cite{jpm3} & GMRES($\widehat{\mathcal{P}}_{I}$) &  $\bold{GMRES(\widehat{\mathcal{P}}_{2})}$\\
\hline
2306     & 41(0.34) & 77(0.79) & 69(0.34)     & $\bold{12(0.09)}$\\
8706     & 49(1.06) & 81(3.72) & 99(2.07)     & $\bold{12(0.32)}$\\
33794   & 49(4.17) & 87(19.3) & 160(9.82)   & $\bold{11(0.73)}$\\
133122 & 61(18.4) & 99(135)  & 242(66.1)   & $\bold{11(2.94)}$\\
\hline
\end{tabular}
\label{table1}
\end{table}

\begin{table}[!b]
\renewcommand\arraystretch{1.5}
\scriptsize
\centering
\caption{\protect\\Comparison of iteration numbers and CPU time for solving Example 1 with $\beta=10^{-4}$}
\begin{tabular}{c|c|c|c|c}
\hline
DoF\verb|\| Method & MINRES($\mathcal{P}^{B}_{D}$) \cite{trh1} & MINRES($\widetilde{\mathcal{P}}$) \cite{jpm3} & GMRES($\widehat{\mathcal{P}}_{I}$) &  $\bold{GMRES(\widehat{\mathcal{P}}_{2})}$\\
\hline
2306     & 117(0.59) & 75(0.75)   & 62(0.35)     & $\bold{22(0.10)}$\\
8706     & 143(3.03) & 90(4.13)   & 85(1.69)     & $\bold{23(0.54)}$\\
33794   & 149(12.5) & 105(23.4) & 127(7.21)   & $\bold{21(1.39)}$\\
133122 & 190(57.1) & 127(171)  & 208(50.0)   & $\bold{21(5.03)}$\\
\hline
\end{tabular}
\label{table2}
\end{table}

\begin{table}[!t]
\renewcommand\arraystretch{1.5}
\scriptsize
\centering
\caption{\protect\\Comparison of iteration numbers and CPU time for solving Example 1 with $\beta=10^{-6}$}
\begin{tabular}{c|c|c|c|c}
\hline
DoF\verb|\| Method & MINRES($\mathcal{P}^{B}_{D}$) \cite{trh1} & MINRES($\widetilde{\mathcal{P}}$) \cite{jpm3} & GMRES($\widehat{\mathcal{P}}_{I}$) &  $\bold{GMRES(\widehat{\mathcal{P}}_{2})}$\\
\hline
2306     & -{}-          & 75(0.72)   & 87(0.29)     & $\bold{50(0.21)}$\\
8706     & 499(10.1) & 97(4.22)   & 97(1.95)     & $\bold{53(1.46)}$\\
33794   & -{}-          & 133(29.7) & 142(8.34)   & $\bold{50(3.34)}$\\
133122 & -{}-          & 174(233)  & 220(52.9)   & $\bold{52(12.7)}$\\
\hline
\end{tabular}
\label{table3}
\end{table}

\begin{table}[!t]
\renewcommand\arraystretch{1.5}
\scriptsize
\centering
\caption{\protect\\Comparison of iteration numbers and CPU time for solving Example 1 with $\beta=10^{-8}$}
\begin{tabular}{c|c|c|c|c}
\hline
DoF\verb|\| Method & MINRES($\mathcal{P}^{B}_{D}$) \cite{trh1} & MINRES($\widetilde{\mathcal{P}}$) \cite{jpm3} & GMRES($\widehat{\mathcal{P}}_{I}$) &  $\bold{GMRES(\widehat{\mathcal{P}}_{2})}$\\
\hline
2306     & -{}- & 81(0.82)   & 180(0.62)   & $\bold{97(0.42)}$\\
8706     & -{}- &103 (4.70) & 256(7.71)   & $\bold{124(3.63)}$\\
33794   & -{}- & 147(31.6) & 275(21.4)   & $\bold{120(9.36)}$\\
133122 & -{}- & 209(269)  & 281(75.5)   & $\bold{125(36.0)}$\\
\hline
\end{tabular}
\label{table4}
\end{table}

\begin{table}[!t]
\renewcommand\arraystretch{1.5}
\scriptsize
\centering
\caption{\protect\\Performance of the $\bold{GMRES(\widehat{\mathcal{P}}_{2})}$ method with "backslash" for solving Example 1}
\begin{tabular}{c|c|c|c|c}
\hline
DoF\verb|\| $\beta$ & $10^{-2}$ & $10^{-4}$ & $10^{-6}$ &$10^{-8}$\\
\hline
2306     & 6(0.08) & 15(0.10)  & 38(0.27)   & $74(0.44)$\\
8706     & 6(0.38) & 14(0.81)  & 37(2.14)   & $91(5.56)$\\
33794   & 5(3.41) & 13(7.94)  & 35(20.3)   & $89(51.2)$\\
133122 & 5(44.2) & 12(95.4)  & 32(242)    & $82(608)$\\
\hline
\end{tabular}
\label{table5}
\end{table}

\begin{table}[!t]
\renewcommand\arraystretch{1.5}
\scriptsize
\centering
\caption{\protect\\Performance of the $\bold{GMRES(\widehat{\mathcal{P}}_{2})}$ method with tolerance $10^{-9}$ for solving Example 1}
\begin{tabular}{c|c|c|c|c}
\hline
DoF\verb|\| $\beta$ & $10^{-2}$ & $10^{-4}$ & $10^{-6}$ &$10^{-8}$\\
\hline
2306     & 13(0.12) & 27(0.12)  & 60(0.35)   & $106(0.50)$\\
8706     & 14(0.38) & 28(0.79)  & 64(1.85)   & $162(6.10)$\\
33794   & 13(0.86) & 26(1.83)  & 62(4.49)   & $153(12.8)$\\
133122 & 13(3.18) & 26(6.25)  & 63(15.8)   & $156(46.1)$\\
\hline
\end{tabular}
\label{table6}
\end{table}

For an illustration of the numerical solution, the computed state and control of the GMRES($\widehat{\mathcal{P}}_{2}$) method are given in Fig. 2. We note that the solutions of different methods for the pure Neumann boundary control problem are the same except for negligible errors. This coincides with the fact that preconditioners improve the efficiency of iterative methods without influence the solutions of linear systems. The iterative results of different methods for solving the pure Neumann boundary control problem with $\beta=10^{-2}, 10^{-4}, 10^{-6}, 10^{-8}$ are listed in Table 1-4. It can be seen from these tables that for a fixed regularization parameter value, the iteration numbers of the MINRES($\mathcal{P}^{B}_{D}$), MINRES($\widetilde{\mathcal{P}}$) and GMRES($\widehat{\mathcal{P}}_{I}$) method grow obviously as DoF increases. By contrast, the iteration number of the GMRES($\widehat{\mathcal{P}}_{2}$) method appears bounded. And it is influenced by the DoF in a much milder way. Moreover, the GMRES($\widehat{\mathcal{P}}_{2}$) method requires much less iterations and CPU time than the other methods. On the other hand,  as the regularization parameter $\beta$ decreases, the iteration number of the GMRES($\widehat{\mathcal{P}}_{2}$) method grows but it still keeps the least. Moreover, the increase is found to occur roughly in a logarithmically linear way. This is similar to the case presented in \cite{kmb1}. Besides, we conducted experiments of the proposed preconditioner where a solve with $Ke$ is done by "backslash" of MATLAB directly. As shown in Table 5, it took much more time even though the required iterations were less. It is noted that the use of AMG method in the proposed preconditioner is further approximation for the coefficient matrix compared with "backslash". We also tested the performance of the GMRES$(\widehat{\mathcal{P}}_{2})$ method with a smaller tolerance $10^{-9}$. The results listed in Table 6 exhibit similar behaviours to the results with tolerance $10^{-6}$.

\subsection{Example 2}\label{sec5-2}
Consider the pure Neumann boundary control problem (\ref{eq2-1})-(\ref{eq2-2}) with the desired state
$$
y_{d}=\left\{
\begin{array}{c}
\begin{array}{ccl}
(2x-1)^{2}(2y-1)^2,&x\leq\frac{1}{2}, y\leq\frac{1}{2},\\
0,&\text{elsewhere},
\end{array}
\end{array}
\right.
$$
which is illustrated in Fig. 3. 

\begin{figure}[!b]
  \centering
  \includegraphics[height=4.5cm]{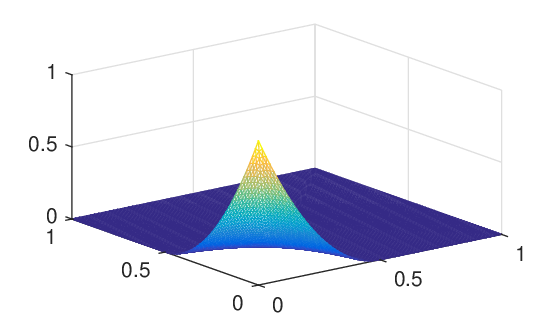}
  \caption{Desired state of Example 2.}
  \label{fig:state2}
\end{figure}

\begin{figure}[!b]
\centering
{\includegraphics[width=\textwidth]{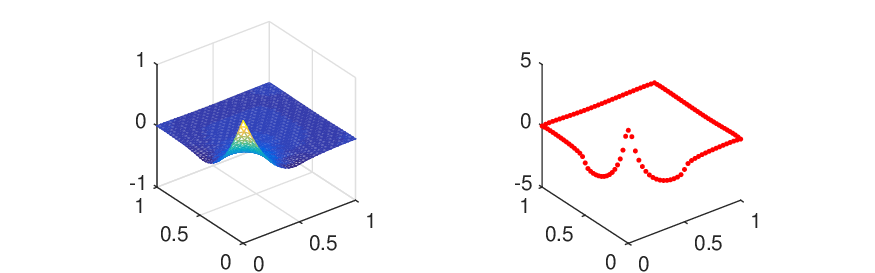}\label{fig:4}}
\caption{Computed state (left) and control (right) of the GMRES($\widehat{\mathcal{P}}_{2}$) method for Example 2 with DoF=2306 and $\beta=10^{-6}$.}
\label{fig:solution2}
\end{figure}

An illustration of the computed state and control of the GMRES($\widehat{\mathcal{P}}_{2}$) method is given in Fig. 4. The iterative results of these methods for solving the pure Neumann boundary control problem with $\beta=10^{-2}, 10^{-4}, 10^{-6},10^{-8}$ are listed in Table 7-10. We observe that the GMRES($\widehat{\mathcal{P}}_{2}$) method requires much less iterations and CPU time than the other methods for the same regularization parameter. The iteration number of the GMRES($\widehat{\mathcal{P}}_{2}$) method still appears bounded. It changes much less than the other methods as DoF increases. Besides, the iteration number of the GMRES($\widehat{\mathcal{P}}_{2}$) method still keeps the least as the regularization parameter $\beta$ becomes smaller. And the iteration increase with $\beta$ is also shown to be a nearly logarithmic increase. The results with "backslash" in Table 11 shows the influence of the AMG method on performance of the proposed preconditioner. For a smaller tolerance $10^{-9}$, Table 12 shows that the GMRES$(\widehat{\mathcal{P}}_{2})$ method has a similar performance with respect to the mesh robustness.

\begin{table}[!b]
\renewcommand\arraystretch{1.5}
\scriptsize
\centering
\caption{\protect\\Comparison of iteration numbers and CPU time for solving Example 2 with $\beta=10^{-2}$}
\begin{tabular}{c|c|c|c|c}
\hline
DoF\verb|\| Method & MINRES($\mathcal{P}^{B}_{D}$) \cite{trh1} & MINRES($\widetilde{\mathcal{P}}$) \cite{jpm3} & GMRES($\widehat{\mathcal{P}}_{I}$) &  $\bold{GMRES(\widehat{\mathcal{P}}_{2})}$\\
\hline
2306     & 39(0.19) & 77(0.76) & 82(0.28)     & $\bold{13(0.07)}$\\
8706     & 51(1.08) & 81(3.64) & 127(2.66)   & $\bold{13(0.30)}$\\
33794   & 51(4.28) & 87(19.2) & 233(16.5)   & $\bold{13(0.86)}$\\
133122 & 63(19.2) & 99(134)  & 425(143)    & $\bold{12(2.91)}$\\
\hline
\end{tabular}
\label{table7}
\end{table}

\begin{table}[!b]
\renewcommand\arraystretch{1.5}
\scriptsize
\centering
\caption{\protect\\Comparison of iteration numbers and CPU time for solving Example 2 with $\beta=10^{-4}$}
\begin{tabular}{c|c|c|c|c}
\hline
DoF\verb|\| Method & MINRES($\mathcal{P}^{B}_{D}$) \cite{trh1} & MINRES($\widetilde{\mathcal{P}}$) \cite{jpm3} & GMRES($\widehat{\mathcal{P}}_{I}$) &  $\bold{GMRES(\widehat{\mathcal{P}}_{2})}$\\
\hline
2306     & 117(0.60) & 73(0.73)   & 66(0.20)     & $\bold{22(0.11)}$\\
8706     & 139(2.91) & 87(3.97)   & 100(2.15)   & $\bold{24(0.60)}$\\
33794   & 144(11.9) & 105(23.2) & 169(10.7)   & $\bold{22(1.44)}$\\
133122 & 177(53.4) & 123(166)  & 302(85.1)   & $\bold{24(5.86)}$\\
\hline
\end{tabular}
\label{table8}
\end{table}

\begin{table}[!b]
\renewcommand\arraystretch{1.5}
\scriptsize
\centering
\caption{\protect\\Comparison of iteration numbers and CPU time for solving Example 2 with $\beta=10^{-6}$}
\begin{tabular}{c|c|c|c|c}
\hline
DoF\verb|\| Method & MINRES($\mathcal{P}^{B}_{D}$) \cite{trh1} & MINRES($\widetilde{\mathcal{P}}$) \cite{jpm3} & GMRES($\widehat{\mathcal{P}}_{I}$) &  $\bold{GMRES(\widehat{\mathcal{P}}_{2})}$\\
\hline
2306     & -{}- & 73(0.71)   & 88(0.27)     & $\bold{47(0.42)}$\\
8706     & -{}- & 95(4.16)   & 111(2.28)   & $\bold{50(1.33)}$\\
33794   & -{}- & 129(28.5) & 182(11.7)   & $\bold{49(3.33)}$\\
133122 & -{}- & 171(228)  & 313(89.6)   & $\bold{51(12.4)}$\\
\hline
\end{tabular}
\label{table9}
\end{table}

\begin{table}[!b]
\renewcommand\arraystretch{1.5}
\scriptsize
\centering
\caption{\protect\\Comparison of iteration numbers and CPU time for solving Example 2 with $\beta=10^{-8}$}
\begin{tabular}{c|c|c|c|c}
\hline
DoF\verb|\| Method & MINRES($\mathcal{P}^{B}_{D}$) \cite{trh1} & MINRES($\widetilde{\mathcal{P}}$) \cite{jpm3} & GMRES($\widehat{\mathcal{P}}_{I}$) &  $\bold{GMRES(\widehat{\mathcal{P}}_{2})}$\\
\hline
2306     & -{}- & 79(0.79)   & 181(0.64)   & $\bold{94(0.41)}$\\
8706     & -{}- & 101(4.61) & 274(8.51)   & $\bold{112(3.29)}$\\
33794   & -{}- & 137(29.4) & 318(26.7)   & $\bold{111(8.47)}$\\
133122 & -{}- & 203(260)  & 374(117)    & $\bold{119(32.8)}$\\
\hline
\end{tabular}
\label{table10}
\end{table}

\begin{table}[!t]
\renewcommand\arraystretch{1.5}
\scriptsize
\centering
\caption{\protect\\Performance of the $\bold{GMRES(\widehat{\mathcal{P}}_{2})}$ method with "backslash" for solving Example 2}
\begin{tabular}{c|c|c|c|c}
\hline
DoF\verb|\| $\beta$ & $10^{-2}$ & $10^{-4}$ & $10^{-6}$ &$10^{-8}$\\
\hline
2306     & 7(0.06) & 15(0.09)  & 35(0.22)   & $68(0.40)$\\
8706     & 7(0.41) & 15(0.84)  & 35(2.03)   & $73(4.42)$\\
33794   & 7(4.51) & 15(9.00)  & 33(18.8)   & $66(38.3)$\\
133122 & 7(58.4) & 15(117)   & 33(249)    & $65(484)$\\
\hline
\end{tabular}
\label{table11}
\end{table}

\begin{table}[!t]
\renewcommand\arraystretch{1.5}
\scriptsize
\centering
\caption{\protect\\Performance of the $\bold{GMRES(\widehat{\mathcal{P}}_{2})}$ method with tolerance $10^{-9}$ for solving Example 2}
\begin{tabular}{c|c|c|c|c}
\hline
DoF\verb|\| $\beta$ & $10^{-2}$ & $10^{-4}$ & $10^{-6}$ & $10^{-8}$\\
\hline
2306     & 14(0.08) & 28(0.13)  & 59(0.26)   & $103(0.48)$\\
8706     & 15(0.36) & 29(0.72)  & 62(1.85)   & $180(6.08)$\\
33794   & 14(0.96) & 28(1.87)  & 61(4.23)   & $203(19.1)$\\
133122 & 14(3.52) & 28(6.84)  & 62(15.6)   & $201(72.6)$\\
\hline
\end{tabular}
\label{table12}
\end{table}

\section{Conclusion}\label{conclusion}
We have presented a new preconditioning method for the saddle point problem arising from the pure Neumann boundary control problem. In this new method, we first regularized the pure Neumann problem by adding a global condition to the candidate solution of the state, which yielded consistent PDE constraints for the original problem and overcame the singularity of the stiffness matrix. Then we extended the saddle point problem to a new one based on this regularization. Furthermore, we transformed the rows of the new saddle point problem to obtain its equivalent form. With this equivalent form, we constructed a block triangular preconditioner based on an approximation of the Schur complement. We have analyzed eigenvalue properties of the preconditioned matrix and given the corresponding eigenvalue bounds. Numerical results showed that the proposed preconditioning method outperformed other methods as it required much less iteration steps and CPU time.

\section{Data Availability Statement}
The data that support the findings of this study are available from the corresponding author, upon reasonable request.

\end{document}